\documentclass[12pt, reqno]{amsart}
\usepackage{amsfonts}
\usepackage[active]{srcltx}
\usepackage{amsmath}
\usepackage{amssymb}
\usepackage{amsthm}
\usepackage[toc,page]{appendix}
\usepackage{breqn}
\usepackage{bbm}
\usepackage{mathtools}
\usepackage{shuffle}
\usepackage{bm} 
\usepackage{upgreek}
\usepackage{graphicx} 
\usepackage{mathrsfs}
\usepackage{fancyhdr}
\usepackage{amsthm}
\usepackage{cases}
\usepackage{amsmath,amssymb,bm}
\usepackage{enumitem}
\usepackage{xcolor}
\usepackage{bigints}
\usepackage{empheq}
\usepackage{geometry}
\usepackage{tikz-cd} 

\usepackage{enumitem}
\usepackage{wasysym}
\usepackage{mathscinet}
\usepackage{verbatim}
\usepackage{color}
\usepackage{graphicx} 
\usepackage[
bookmarks=true,         
bookmarksnumbered=true, 
colorlinks=true, pdfstartview=FitV, linkcolor=blue, citecolor=blue,
urlcolor=blue]{hyperref}
\usepackage{microtype}
\usepackage{amsmath}
\usepackage{amssymb}
\usepackage{amsthm}
\usepackage{mathtools}
\usepackage{shuffle}
\usepackage{bm} 
\usepackage{upgreek}
\usepackage{mathrsfs}
\usepackage{fancyhdr}
\usepackage{amsthm}
\usepackage{cases}
\usepackage{amsmath,amssymb,bm}
\usepackage{enumitem} 
\usepackage{xcolor}

\theoremstyle{plain}

\newtheorem{theorem}{Theorem}[section]

\newtheorem{lemma}[theorem]{Lemma}
\newtheorem{problem}[theorem]{Problem}

\renewenvironment{proof}[1][Proof]{\textbf{#1.} }{\ \rule{0.5em}{0.5em} \par }
\theoremstyle{remark}

\theoremstyle{definition}

\def\lbar#1{\lfloor{#1\rfloor}}

\def\111{{\mathbb{I}}}

\def\RR{\mathbb{R}}\def\TT{\mathbb{T}_t}
\def\EE{\mathbb{E}}
\def\NN{\mathbb{N}}

\def\la{{\lambda}}

\def\De{{\Delta}}

\def\Om{{\Omega}}

\def\al{{\alpha}}

\def\De{{\Delta}}

\def\tr{{ \hbox{ Tr} }}

\def\la{{\lambda}}

\def\al{{\alpha}}

\newcommand{\JJ}{\mathcal J}

\let\Section=\section
\def\section{\setcounter{equation}{0}\Section}

\usepackage{graphicx} 

\date{}
\title[Logarithmic Euler-Maruyama scheme]{logarithmic Euler-Maruyama scheme  for multi-dimensional stochastic delay equations with jumps}

\author[Agrawal]{Nishant Agrawal} 
\address{Department of Mathematical and Statistical Sciences \\
 University of Alberta at Edmonton \\
Edmonton,  Canada, T6G 2G1}
\email{nagrawal@ualberta.ca, yaozhong@ualberta.ca}

\author[Hu]{Yaozhong Hu}



\begin{document}
\maketitle
\section{Introduction}
 In \cite{rpaper} we introduced a logarithmic   Euler-Maruyama scheme for a
 single stochastic delay equation, which preserve positivity if the solution to the original equation is positive.
The convergence rate was  also obtained for such scheme.
This  scheme is important for simulation of the paths of the 
equation. It plays important role in  option pricing for example since
we often cannot obtain the explicit pricing value and we need to
use Monte-Carlo to complete the evaluation. Naturally our next question would be what will be the
analogous  scheme for a system of stochastic delay equations and if such  schemes converges.
This type of problems is  very important   since there is always more than one stock in the real  market.
Now in more than one dimension, the problem of positive solution
and the numerical schemes which preserve the positivity are much more complicated.
 In this paper 
we shall extend 
our work in \cite{rpaper} to  a system of stochastic delay differential equations. The 
problems of existence and uniqueness of a  positive
are solved. The multi-dimensional logarithmic  
Euler-Maruyama scheme are constructed which preserve the positivity of the 
approximate solutions. The scheme is proved to be convergent with rate $0.5$.
\section{Positivity}
We consider the following system of stochastic delayed differential equations: 
  \begin{empheq}[left=\empheqlbrace]{align }
       dS_i(t)&=    \sum_{j=1}^d f_{ij} (S (t-b))  
S_j (t) dt \nonumber\\
& \qquad\qquad  +S_i(t-)  \sum_{j=1}^d g_{ij} ( S (t-b)) dZ_j(t),, \quad i=1, \cdots, d\,,
\label{e.2.1} \\
S_i(t) &=  \phi_i(t) \,,\quad t\in [-b, 0]\,, \ i=1, \cdots, d\,,  \nonumber
 \end{empheq}
where  $S(t)=(S_1(t), \cdots, S_d(t))^T$  and 
\begin{enumerate}
\item[(i)] 
 $f_{ij}, g_{ij}:\RR^d\rightarrow
\RR$ are some given bounded measurable functions for all$ \hspace{2mm} 0 \leq i,j \leq d$.
\item[(ii)]  $b>0$  is  a given number representing the delay of the equation.
\item[(iii)] $\phi_i:    [-b,0]\rightarrow \RR$ is a 
(deterministic) measurable function for all$ \hspace{2mm} 0 \leq i \leq d$.
\item[(iv)] 
$Z_j(t)=\sum_{k=1}^{N_j(t)}Y_{j,k}$ are L\'evy processes,  where $Y_{j,k}$ are i.i.d random variables, \ 
$N_\ell (t)$ are independent Poisson random processes which are also independent of 
$Y_{j,k}$  for $j, \ell, =1,2, \cdots, d$,
$k  =1,2, \cdots$ 
\end{enumerate} 
 Let $|\cdot|$  
Euclidean   norm in $\RR^d$.  If $A$ is $d\times m$ matrix, we denote
\[
|A|=\sup_{|x|\le 1} |Ax|\,.
\]
For example, if $A=I+M$ is a $d\times d$ matrix, where $M=(m_{ij})_{1\le i,j\le d}$
is a matrix, then we can bound the norm of $A$ as follows.
Let $ 0\le \la_1\le  \cdots\le \la_d $ be   eigenvalues of $M^TM$
(since $M^TM$ is a positive definite matrix, we can assume that its eigenvalues are all positive). 
Then 
\begin{eqnarray*}
| I+M |
&=&\sup_{|x|\le 1} 
 \sqrt{|x|^2+x^TM^TMx} \le \sqrt{ 1+\max_{1\le i\le d} \la_i   }|x|\\
 &\le& \sqrt{ 1+\sum_{i=1}^d \la_i   }|x|\,. 
\end{eqnarray*}
But $\sum_{i=1}^d \la_i=\tr (M^TM)$. Thus we have
\begin{equation}
| I+M | 
 \le  \sqrt{ 1+\tr (M^TM) }|x|\,.  
\end{equation}

To study the above stochastic differential equation, it is common to introduce the
Poisson random measure associated with the L\'evy process $Z_j(t)$. We write the jumps of the process $Z_j$ at time $t$  by
   \[
    \Delta Z_j(t):= Z_j(t) - Z_j(t-) \quad \hbox{if $ \Delta  Z_j(t)\not=0$} \hspace{3mm} j=1,2, \cdots,  d\,.
    \]
 Denote  $\mathbb{R}_0 := \mathbb{R} \backslash \{0\}$ and let $\mathcal{B}(\mathbb{R}_0)$ be the Borel $\sigma$-algebra generated by the family of
 all Borel subsets $U \subset \mathbb{R}$, such that $\Bar{U} \subset \mathbb{R}_0$. For any 
 $t>0$ and for any  $U \in \mathcal{B}(\mathbb{R}_0)$ 
  we   define the {\it Poisson random measure},  
 $N_j: [0, T]\times \mathcal{B}(\mathbb{R}_0)\times \Om\rightarrow \RR$ (without confusion we use the same notation $N$),  associated with 
 the L\'evy process $Z_j(t)$ by 
\begin{equation}
 N_j(t, U) := \sum_{0 \leq s \leq t, \ \Delta Z_j(s)\not =0}\chi_U(\Delta Z_j(s)), \hspace{5mm} j=1,2,\cdots,d,     \end{equation} 
 where $\chi_U$ is the indicator function of $U$.
 The associated  L\'evy measure $\nu$ of the L\'evy process $Z_j$ is  given  by
 \begin{equation}
 \nu_j(U) := \mathbb{E}[N_j(1,U)] \hspace{10mm} j=1,2,\cdots, d.
 \end{equation}  
We now define the compensated Poisson random  measure $\tilde{N}_j$ associated with
the L\'evy process $Z_j(t)$ by 
 \begin{equation}
 \Tilde{N}_j(dt,dz) := N_j(dt,dz) - \EE\left[ N_j(dt,dz)  \right] = N_j(dt,dz) - \nu_j(dz)dt\,. 
 \end{equation}
We    assume that the process $Z_j(t)$ 
has only bounded negative jumps to guarantee that the solution $S(t)$ to \eqref{e.2.1} is positive. This means that there is an interval $\JJ=[-R, \infty)$ bounded from the left 
such that $\Delta Z_j(t)
\in \JJ$ for all $t>0$ and for all $j=1,2, \cdots d$. 

With these notations,  we can write
\[
 {Z}_j(t) = \displaystyle\int_{ [0,t] \times \JJ } z  {N}_j (ds,dz) \quad {\rm or}\quad d {Z}_j(t) = \displaystyle\int_{ \JJ } z  {N}_j (dt,dz)  
 \]
and  write \eqref{e.2.1} as 
\begin{empheq}[left=\empheqlbrace]{align }
 dS_i(t)&= \sum_{j=1}^d f_{ij} (S(t-b))  
S_j (t) dt +S_i(t-)\sum_{j=1}^d \displaystyle\int_{\JJ}zg_{ij} ( S (t-b))\nu_j(dz)dt\nonumber\\
& \qquad\qquad +S_i(t-) \sum_{j=1}^d \displaystyle\int_{\JJ}zg_{ij} ( S (t-b)) \tilde{N}_j(dz,dt)\,, \nonumber\\
S_i(t) &= \phi_i(t) \,,\quad t\in [-b, 0]\,, \ i=1, \cdots, d\,. \label{e.2.5} 
 \end{empheq}
In fact we can consider a slightly more general version of system of equations than \eqref{e.2.5}:
\begin{empheq}[left=\empheqlbrace]{align }
 dS_i(t)&=    \sum_{j=1}^d f_{ij} (S(t-b))  
S_j (t) dt \nonumber\\
& \qquad\qquad +S_i(t-)  \sum_{j=1}^d \displaystyle\int_{\JJ}g_{ij} (z, S (t-b)) \tilde{N}_j(dz,dt),\quad i=1, \cdots, d\,,
\nonumber\\
S_i(t) &=  \phi_i(t) \,,\quad t\in [-b, 0]\,, \ i=1, \cdots, d\, .    
\label{e.2.6} 
 \end{empheq}
First, we discuss  the existence, uniqueness and positivity  
of  \eqref{e.2.6}. 

\begin{theorem} \label{t.2.1}  
Suppose that $f_{ij}
:\RR^d \rightarrow \RR$  and $  g_{ij}:\JJ\times \RR^d\rightarrow \RR\,, \ 
1\le i,j\le d$ are  bounded measurable  functions such that  there is a 
constant $\al_0>1$ satisfying 
$g_{ij}(z, x) \ge \al_0>-1$ for all $
1\le i,j\le d$, for all $z\in \JJ$ and for all 
$ x\in \RR$,  where $\JJ=[-R, \infty)$ is the common supporting set of the Poisson measures $\tilde N_{j}(t, dz), j=1, \cdots, d$.  
If for all $i\not= j$, $f_{ij}(x)\ge 0$ for all $x\in \RR$, 
and $\phi_i(0)\ge 0\,, \ i=1, \cdots, d$, then,  the stochastic differential delay equation \eqref{e.2.6} 
admits a unique pathwise solution
such  that    $S_i(t) \geq 0$   almost surely for all $i=1, \cdots, d$  and for all $t > 0$.   
  \end{theorem} 
\begin{proof} The theorem is stated and proved in  \cite[Theorem 1]{rpaper}  following  the method of \cite{humulti1} 
(where the case of Brownian motion was dealt with).  In fact,  the existence and uniqueness are routine and  easy.  The main  point is to show the positivity of the solution. The idea
in \cite{rpaper}  was to decompose the solution to \eqref{e.2.6} as product of some nonnegative  processes.  Here we give a slightly different     decomposition which will prove the positivity and will be very useful in our numerical scheme.

Denote   $\tilde f_{ij}(t)=f_{ij}(S(t-b))$ and   $\tilde g_{ij}(t,z)=g_{ij}(z,S(t-b))$.   
Let $Y_i(t)$ be the solution to the stochastic differential equation
\[
dY_i(t) =   \tilde f_{ii} (t)  
Y_i(t) dt +  Y_i(t-)  \sum_{j=1}^d \int_{\JJ} \tilde g_{ij} (t,z)   \tilde N_{j} (dt, dz) 
\]
with initial conditions $Y_i(0)=\phi_i(0)$.  Since this is a scalar equation for $Y_i(t)$, its explicit solution can be represented 
\begin{eqnarray}
 Y_i(t) &=&\phi_i(0)\exp\bigg\{  
 \sum_{j=1}^d    \log\left[  1+\tilde g_{ij} 
 (s,z ) \right] \Tilde{N}_j(ds,dz)+\int_0^t 
  \tilde f_{ii}(s)   ds\nonumber\\
  &&\qquad  + \sum_{j=1}^d \int
  _{[0, t]\times \JJ} \Big( \log\left[1+\tilde 
  g_{ij}(s,z)\right] -\tilde g_{ij}(s, z ) \Big)ds\nu_j(dz)     \bigg\}\,,  
  \label{e.2.7a} 
 \end{eqnarray}
 where $\nu_j$ is the associated L\'evy measure for $\tilde N_j(ds, dz)$.     Let 
$p_i(t)$ be the solution to the following system of 
equations
\[
dp_i(t) =\sum_{j=1, j\not=i}^d \tilde f_{ij}(t)p_j(t) dt\,, \quad 
p_i(0)=1\,, \quad i=1, \cdots, d\,.  
\]
Since by the assumption    that  $\tilde 
f_{ij}(t)\ge 0$ almost surely for  all $i\not=j$, Theorem 
 \cite[p.173]{matanlysis}   implies 
that
$p_{i}(t)\ge 0$ for all $t\ge 0$  almost surely. 
Now it is easy to check by the It\^o formula that $\tilde S_i(t)=p_i(t)Y_i(t)$  satisfies 
\eqref{e.2.6}   and 
$\tilde S_i(t)\ge 0$ almost surely. By the uniqueness of the solution 
we see that $S_i(t)=\tilde S_i(t)$ for $i=1, \cdots, d$.  The theorem is then proved. 
\end{proof}

 \section{Convergence rate  of logarithmic Euler-Maruyama scheme} 
In this section we construct numerical scheme to approximate \eqref{e.2.1} by positive value processes. 
   
Motivated by the proof of Theorem \ref{t.2.1} we shall decompose equation \eqref{e.2.1} into the following system:
\begin{subequations} 
  \begin{empheq}[left=\empheqlbrace]{align }
dX_i(t) &= f_{ii}((S(t-b)))
X_i(t )dt+X_i(t-) \sum_{j=1}^d g_{ij} (S(t-b)) dZ_{j} (t)\label{e.3.1a}\\ 
 dp_i(t)&=\sum_{j=1,   j\not=i}^df_{ij}((S(t-b)))p_j(t)dt,   \label{e.3.1b} \\ 
S_i(t)& = p_i(t)\cdot X_i(t)\,, \hspace{3mm} i=1,2, \cdots,d\,. 
\label{e.3.1c}  
 \end{empheq}
\end{subequations}
The reason is, as in  
the proof of Theorem \ref{t.2.1}, that $X_i(t)$ and $p_{i}(t)$ are all positive.

Consider a finite time interval $[0, T]$ for some fixed $T>0$  and let    $\pi$ be a  partition  of the time interval $[0, T]$:
\[
\pi: 0=t_0<t_1<\cdots <t_n=T\,.
\]  
Let $\Delta_k  =  
 t_{k+1}-t_k  $  and $\Delta =\max_{0\le k\le n-1} 
 (t_{k+1}-t_k) $ and assume  $\De<b$.


We shall now construct explicit logarithmic 
Euler-Marauyama recursive scheme to numerically 
solve \eqref{e.3.1a}-\eqref{e.3.1c}. 
%
By the expression \eqref{e.2.7a} 
the solution $X$ on $[t_k, t_{k+1}]$ to Equation \eqref{e.3.1a} 
%
is given\\
 by 
\begin{eqnarray*}
 X_i(t) &=&X_i(t_k) \exp\bigg\{ \int_{t_k}^t 
 f_{ii}(S(s-b)) ds+\sum_{j=1}^d
 \int_{t_k}^t \log\left[ 1+ g_{ij} 
 (S(u-b) ) dZ_j(s) 
 \right] \bigg\}\,, 
 \end{eqnarray*}
 where $Z_j(t):=\sum_{k=1}^{N_j(t)}Y_{j,k}$.  
  If  we denote by  $F(x ) $  the $d\times d$ matrix whose diagonal elements are all zero and whose off diagonal entries are  $f_{ij}(x)$,   namely,
\[
 F_{ij}(x )  =\begin{cases}
 0&\qquad \hbox{when $i=j$}\\
 f_{ij}(x) &\qquad \hbox{when $i\not=j$}\,. 
 \\
 \end{cases} 
\]
With this notation  we can write 
    \eqref{e.3.1b}   as a matrix form: 
\begin{equation}
\frac{dp(t)}{dt}=F((S(t-b)))p(t)\,, 
\quad p(t)=(p_1(t), \cdots, p_d(t))^T\,, 
\label{e.3.3} 
\end{equation} 
and its solution on the sub-interval $[t_k, t_{k+1}]$ is given by   
\begin{equation} 
p(t) 
=  \exp\Big(   \tilde F(S(t-b) ) 
   \Big) p(t_k)\,, \hspace{5mm} t\in [t_k, t_{k+1}]\,,   \label{ndim3} 
\end{equation}
where the exponential of a matrix is in the usual sense: $e^A =\sum_{k=0}^\infty A^k/k!$, 
  the integral of a matrix is  
entry-wise.  Here  due to the noncommutativity  $ \tilde F(S(t-b))$ is complicated to determine
and we give the following formula for the sake of completeness:  
\begin{eqnarray}
 \tilde F(S(t-b))
 &=&\sum_{r=1}^\infty \sum_{\sigma\in P_r}\left(
\frac{(-1)^{e(\sigma)}}{r^2\left({r-1}\atop{e(\sigma)}\right)}\right) \int_{T_r(t)}  
\label{e.3.4a} \\
&&\quad  \times [[\cdots [F(S(u_{\sigma(1)}-b) F(S(u_{\sigma(2)}-b)]
\cdots] F(S(u_{\sigma(r)}-b)]du_1 \cdots du_r \nonumber
\end{eqnarray}
 is given by the Campbell-Baker-Hausdorff-Dynkin Formula
 (see e.g. \cite{hu}, \cite{St}),  where 
 $P_r$ is the set of all permutations of $\{1, 2, \cdots, r\}$,
 $e(\sigma)$ is the number of errors in ordering consecutive terms
 in $\{\sigma(1), \cdots, \sigma(r)\}$, $[AB]=AB-BA$ denotes the commutator
 of the matrices, and $T_r(t)=\left\{ 0\\
 <u_1<\cdots<u_r<t\right\}$. 
 
Analogously to  \cite{rpaper} we propose  the following logarithmic scheme to approximate 
the  solution:  
\begin{subequations} 
  \begin{empheq}[left=\empheqlbrace]{align }
X_i^{\pi}(t) 
 &=  X_i^{\pi}(t_k) \exp\Big( f_{ii}( S^{\pi}(t_k-b))
 (t-t_k) \nonumber\\&+\sum_{j=1}^d
 \ln{\Big(1 + g_{ij}(S^{\pi}(t_k-b))(Z_j(t)- Z_{j}(t_k)\Big)}  \Big)\,,   \label{e.3.6a}
\\
p^{\pi}(t ) 
 &=   \Big[ F(S^{\pi}(t_k-b)) )(t-t_k) + I \Big]p^{\pi}(t_k),     \label{e.3.6b}\\
 S^{\pi}_i(t) &=  p^{\pi}_i(t)X^{\pi}_i(t)\,,
\label{e.3.6c}\\ 
 X_i^\pi(0)&=\phi_i(0)\,,  \quad 
 p^\pi(0)={\bf 1}\,,\qquad  t_k\le t\le t_{k+1}\,, \hspace{7mm} k=1,2, \cdots, n-1\,. 
 \label{e.3.6d}
  \end{empheq}
\end{subequations}   

We introduce step processes 
 \begin{eqnarray*}
\begin{cases} v_{1}(t) = \sum_{k=0}^{\infty}\mathbbm{1}_{[t_k , t_{k+1}
)}(t)S^\pi (t_k )\\ 
 v_{2}(t) = \sum_{k=0}^{\infty}\mathbbm{1}_{[t_k , t_{k+1}
 )}(t)S^\pi (t_k -b ).
 \end{cases} 
 \end{eqnarray*} 
  Using the above step process we can write the continuous interpolation for $X_i$ as 
 \begin{eqnarray}\label{a24}
&& X^{\pi}_i(t) = \exp\Big(\displaystyle\int_0^t f_{ii}( v_2(u))du+\sum_{j=1}^d\sum_{0 \leq u \leq t, \De Z(u) \neq 0} \ln\Big(1+ g_{ij}(v_2(u))Y_{j,N_j(u)} \Big) \Big) \,. 
\end{eqnarray}
Denote $\lfloor t \rfloor=\max\{k, t_k<t\}$.  
From \eqref{e.3.6b} we have   
 \begin{eqnarray} 
p ^{\pi}(t) 
 &=& \Big[ \displaystyle\int_{t_{\lfloor t\rfloor}}^t F(v_2(u) )du + I \Big]\prod_{k=1}^ {\lfloor t\rfloor} \Big[ \displaystyle\int_{t_{k-1}}^{t_{k}}F(v_2(u))du  + I \Big]\,. 
  \label{ndim3} 
 \end{eqnarray}
%
We first show that $p^{\pi}(t_k)\geq 0$.
\begin{lemma}
If $\phi(0)\ge 0$ a.s.,  then  $p^\pi(t_k)\ge 0 $ a.s.
with $p^{\pi}(t) = \phi(t)$ for all $t \in [- b, 0]$. 
\end{lemma}
\begin{proof} This can be seen 
from \eqref{e.3.6b} and by induction.  Assume 
$p^\pi(t_k)\ge 0 $ a.s.   Since by 
our definition of $F(S^{\pi}(t_k-b))$ we know all
of its components are positive,
we see from \eqref{e.3.6b}  that 
$p^\pi(t )\ge 0 $ a.s.  for all $t_k\le t\le t_{k+1}$.
\end{proof}

Similarly we will have
\begin{lemma}
If $\phi(0)\ge 0$ a.s.,  then  $X^\pi(t)\ge 0 $ a.s. , hence $S^\pi(t)\ge 0$ a.s. for all $0\le t\le T$.  
\end{lemma}

To obtain the convergence of the logarithmic Euler–Maruyama scheme 
\eqref{e.3.6a}-\eqref{e.3.6d}, we make the following assumptions:
  \begin{enumerate}
 \item[{\bf (A1)}] The initial data $\phi_i(0)>0$ and it is H\"older continuous,  i.e.  there exist constant $\rho >0$ and $\gamma \in [1/2,1)$ such that for $t,s \in [-b,0]$
 \begin{eqnarray}
 |\phi_i(t) - \phi_i(s)| \leq \rho |t-s|^{\gamma}.\hspace{2mm} i=1,2, \cdots, d.
 \end{eqnarray}
%
%
%
 \item[{\bf (A2)}] $f_{ij}$ is bounded. $f_{ij}$ and $g_{ij}$ are  global Lipschitz for $i, j=1,2,\cdots,d$. This means that there exists a constant  $\rho>0$ such that
 \begin{eqnarray}
\begin{cases}\Big|g_{ij}(x_1)-g_{ij}(x_2)\Big| 
 \leq \rho |x_1 - x_2|\hspace{5mm}\forall \  x_1, x_2\in \RR^d\, \,;\\ 
 \Big|f_{ij}(x_1)-f_{ij}(x_2)\Big| 
 \leq \rho  |x_1 - x_2|\,\,,\quad \forall \ x_1, x_2\in \RR^d\,;
 \nonumber \\
  \big|f_{ij}(x)\big|  \leq \rho  \,,\quad \forall x\in \RR^d.
 \end{cases} \label{pema102}
 \end{eqnarray} 
 \item[{\bf (A3)}] The support  $\JJ $ of the Poisson random measure $N_j$ (associated with $Z$) is contained in $[-R, \infty)$ for each $j=1,2,\cdots,d$ for some $R>0$ and   there are   
constants  $\al_0>1$  and $\rho>0$ satisfying 
$-\rho\le g_{ij}(x) \le \frac{\al_0}{R} $ for all   $ x\in \RR^d$ and for all $i,j=1,2, \cdots,d$. 
 \item[{\bf (A4)}] For any $q>1$  there is a $\rho_q>0$
 \begin{eqnarray}&&
 \displaystyle\int_{\JJ}(1+ |z |)^q\nu_i(dz) \leq \rho_q \,,  \hspace{3mm}\forall x  \in \mathbb{R}\,. \hspace{2mm} i=1,2, \cdots,d. \label{pema103}
 \end{eqnarray}
 \end{enumerate} 
\begin{lemma}{\label{a19}}
Let   Assumptions (A1)--(A4) be satisfied.   Then,   for any $q\ge 1$, there exists $K_q$, independent of the partition $\pi$, 
 such that
 \begin{eqnarray*}
 \mathbb{E}\Big[\sup_{1\le i\le d} \sup_{0 \leq t \leq T}|X_i(t)|^q \Big] \vee \mathbb{E}\Big[\sup_{1\le i\le d}\sup_{0 \leq t \leq T}|X_i^\pi(t) |^q \Big]
 \leq K_q.
 \end{eqnarray*} 
\end{lemma}
\begin{proof} From our definition of $X_i^\pi$ and boundedness of $f_{ij}$ for all $i,j$ we have 
\begin{eqnarray} 
 \mathbb{E}\Big[\sup_{0 \leq t \leq T}|X_i^\pi(t) |^q \Big] 
 \nonumber 
 &= &
 \mathbb{E}\Big[\sup_{0 \leq t \leq T} \exp\Big( q\displaystyle\int_0^t f_{ii}( v_2(u))du+q\sum_{j=1}^d\sum_{0\leq u\leq t, \Delta Z(u)\neq 0}\ln(1+g_{ij}(v_2(u)) Y_{j,N_j(u)}) \Big)\Big]
 \nonumber\\
 & = &  \mathbb{E}\Big[\sup_{0 \leq t \leq T}\exp\Big(q\displaystyle\int_0^t f_{ii}( v_2(u))du +q\sum_{j=1}^d\displaystyle\int_{ \mathbb{T} }\ln(1+z_jg_{ij}(v_2(u)))N_j(du,dz)\Big)\Big]\nonumber \\
& \leq &  K\mathbb{E}\Big[\sup_{0 \leq t \leq T}\exp\Big(q\sum_{j=1}^d\displaystyle\int_{ \mathbb{T} }\ln(1+z_jg_{ij}(v_2(u)))N_j(du,dz)\Big)\Big]\nonumber 
 \\ &=:&KI\,,\label{a18}
 \end{eqnarray}
where 
 $\TT=[0, t]\times \JJ$.   
Denote  $h_j = ((1+z_jg_{i,j}(v_2(u))^{2q}-1))/z_j $.   Then,     
 \begin{eqnarray*}
 I&=&
 \mathbb{E}\Big[\sup_{0 \leq t \leq T}\exp\Big(\frac12\sum_{j=1}^d \displaystyle\int_{ \mathbb{T}_t }\ln(1+z_jh_j)N_j(du,dz_j) \Big)\Big] \\
 &=&
 \mathbb{E}\Big[\sup_{0 \leq t \leq T}\exp\Big(\sum_{j=1}^d\Big(\frac12 \displaystyle\int_{ \mathbb{T}_t }\ln(1+z_jh_j)\tilde N_j(du,dz_j) 
 +\frac12 \int_{ \mathbb{T}_t }\ln(1+z_jh_j)\nu_j(dz_j) du 
 \Big)\Big)\Big] \\
 &=&
 \mathbb{E}\Big[\sup_{0 \leq t \leq T}\exp\Big(\sum_{j=1}^d\Big(\frac12 \displaystyle\int_{ \mathbb{T}_t }\ln(1+z_jh_j)\tilde N_j(du,dz_j) 
 +\frac12 \int_{ \mathbb{T}_t }\left[ \ln(1+z_jh_j)-z_jh_j\right] 
 \nu_j(dz_j) du 
 \Big)\Big)\Big] \\
 &&\qquad \sup_{0 \leq t \leq T}\exp\Big(\sum_{j=1}^d 
 -\frac12 \int_{ \mathbb{T}_t } (1+z_jg_{ij}(v_2(u))^{2q}-1)\ 
 \nu_j(dz_j) du 
 \Big)\Big] \\
 &\le &C_q
 \mathbb{E}\Big[\sup_{0 \leq t \leq T}\exp\Big(\sum_{j=1}^d\Big(\frac12 \displaystyle\int_{ \mathbb{T}_t }\ln(1+z_jh_j)\tilde N_j(du,dz_j) 
 +\frac12 \int_{ \mathbb{T}_t }\left[ \ln(1+z_jh_j)-z_jh_j\right] 
 \nu_j(dz_j) du 
 \Big)\Big)\Big] \,, 
 \end{eqnarray*}
where we used   Assumption (A4)  and the  boundedness of $g_{ij}$.
Write for $k=1,2, \cdots,d$
 \[
 M_{k,t}:=\exp\Big( \displaystyle\int_{ \mathbb{T}_t }\ln(1+z_kh_k)\tilde N_k(du,dz_k) 
 + \int_{ \mathbb{T}_t }\left[ \ln(1+z_kh_k)-z_kh_k\right] 
 \nu_k(dz_k) du 
 \Big)
 \,.
 \]
Then   $(M_{k,t}, 0\le t\le T)$ is an exponential martingale.
Now an application of the Cauchy--Schwartz inequality 
yield  
\begin{eqnarray*}
 I
 &\le & C_q
 \bigg\{\mathbb{E}\Big[\sup_{0 \leq t \leq T}M_{1,t} \Big]\bigg\}^{d/2} \,, 
 \end{eqnarray*} 
 which proves 
$$\mathbb{E}\Big[\sup_{0 \leq t \leq T}|X_i^\pi(t) |^q \Big] \le
K_q<\infty.$$   In the same way,  we can show 
$\mathbb{E}\Big[\sup_{0 \leq t \leq T}|X_i (t) |^q \Big] \le
K_q<\infty$.   This completes the proof of the lemma.   
\end{proof}
\begin{lemma} \label{a26}
Assume Assumptions (A1)--(A4).
Then for $\De<1$,   there is a constant $K>0$,
independent of $\pi$, such that
 \begin{eqnarray*}
 \mathbb{E}
  \sup_{0\le t\le T} \Big|S^\pi(t)- v_2(t) \Big|^{p} \leq K \Delta^{p/2} \, .   
 \end{eqnarray*} 
 \end{lemma}
 \begin{proof}
Let $\lfloor t\rfloor =t_k
$ if $t \in [t_k,t_{k+1})$ for some $k $. We have $v_2=(v_{21},v_{22},\cdots,v_{2d})$ for which we write in short $v_2=(\bar{v}_{1},\bar{v}_{2},\cdots,\bar{v}_{d})$. For any $i=1, \cdots, d$,  \begin{eqnarray}{\label{a21}}
 &&\mathbb{E}\sup_{0\le t\le T} \Big| S_i^{\pi}(t) -\bar{v}_i(t)\Big|^p
 = \mathbb{E}
 \sup_{0\le t\le T} \Big|p_i^\pi(t)X_i^\pi(t)-p_i^\pi(\lbar{t})X_i^\pi(\lbar{t})\Big|^p\ \nonumber\\
 &&= \mathbb{E}
 \sup_{0\le t\le T} \Big|p_i^\pi(t)X_i^\pi(t)-p_i^\pi(\lbar{t})X_i^\pi(t)+p_i^\pi(\lbar{t})X_i^\pi(t)-p_i^\pi(\lbar{t})X_i^\pi(\lbar{t})\Big|^p\nonumber\\
 && \leq C \Big(\mathbb{E}\sup_{0\le t\le T} \Big|p_i^\pi(t)-p_i^\pi(\lbar{t})\Big|^{2p}\Big)^{1/2}\Big(\mathbb{E}\sup_{0\le t\le T} \Big|X_i^\pi(t)\Big|^{2p}\Big)^{1/2}\\
 &&\qquad \quad +C\Big(\mathbb{E}\sup_{0\le t\le T}\Big|X_i^\pi(t)-X_i^\pi(\lbar{t})\Big|^{2p}\Big)^{1/2}\Big(\mathbb{E}\sup_{0\le t\le T}\Big|p_i^\pi(\lbar{t})\Big|^{2p}\Big)^{1/2}.\nonumber
\\
&& \label{a33}\end{eqnarray}
By Assumption 2 we can bound $\mathbb{E}\sup_{0\le t\le T}\Big|p_i^\pi(\lbar{t})\Big|^{2p}$ and by lemma \eqref{a19} we can bound $\mathbb{E}\sup_{0\le t\le T}\Big|X_i^\pi(t)\Big|^{2p}$.
We now bound the other two components.
\begin{eqnarray}
&& \mathbb{E}
\sup_{0\le t\le T} \Big|p_i^\pi(t)-p_i^\pi(\lbar{t})\Big|^{2p} \leq \sum_{\substack{j,  j\neq i}}^d\mathbb{E}\sup_{0\le t\le T} \Big|\int_{\lbar{t} }^tf_{ij}(v_2(u))du\Big|^{2p}.
\end{eqnarray}
By Assumption 2 it is easy to see that  for some constant $C_1$ \begin{eqnarray}
&& \mathbb{E}\sup_{0\le t\le T} \Big|p_i^\pi(t)-p_i^\pi(\lbar{t} )\Big|^{2p} \leq C_1 \De^{2p}. \label{a34}
\end{eqnarray}
For $\mathbb{E}\sup_{0\leq t\leq T}\Big|X_i^\pi(t)-X_i^\pi(\lbar{t})\Big|^{2p}$ we use the expression for $X_i^\pi(t)$, boundedness of $f_{ij}$ for all $i,j$ and 
use $|e^x-e^y|\leq |e^x+e^y||x-y|$  to obtain 
\begin{eqnarray}
 \mathbb{E}
 \sup_{0\le t\le T}\Big|X_i^\pi(t)-X_i^\pi(\lbar{t})\Big|^{2p} &\leq& \left\{\EE \sup_{0\le t\le T} \Big| X_i^\pi(t)+X_i^\pi(\lbar{t} ) \Big|^{2p}\right \}^{1/2}\nonumber \\
 &&\cdot K \left\{\EE \sup_{0\le t\le T}\left[ \Big| \sum_{j=1}^d\sum_{\lbar{t}\le s<t}\ln(1+g_{ij}(v_2(s))Y_{j,N(s)} )\Big|\right]^{2p}\right \}^{1/2}. \nonumber
\end{eqnarray}
%
The first factor is bounded and now, we want to bound the second factor: 
\[
I:=\EE\sup_{0\le t\le T} \left| \sum_{j=1}^d\sum_{\lbar{t} \leq s \leq t}\ln(1+g_{i,j}(v_2(s))Y_{j,N_j(s)}) \right|^{2p} \,.
\]
(We use the same 
notation $I$ to denote different quantities in different 
occasions and this does  not cause ambiguity).   
We write the above sum as an integral: 
 \begin{eqnarray*}
I
&=&\mathbb{E}\sup_{0\le t\le T} \Big|\sum_{j=1}^d\displaystyle\int_{\JJ }\displaystyle\int_{\lbar{t} } ^t\ln(1+z_jg_{ij}(v_2(s))) {N_j}(ds,dz_j)\Big|^{2p}\\ 
&=&\mathbb{E}\sup_{0\le t\le T}\Big|\sum_{j=1}^d\displaystyle\int_{\JJ }\displaystyle\int_{\lbar{t}} ^t\ln(1+z_jg_{ij}(v_2(s)))\tilde{N_j}(ds,dz_j)\\
&&+\sum_{j=1}^d\displaystyle\int_{\JJ }\displaystyle\int_{\lbar{t}} ^t\ln(1+z_jg_{ij}(v_{2}(s)))\nu_j(dz_j)ds\Big|^{2p}\\
&\le& C_p \left(\De^{2p} + \mathbb{E}\sup_{0\le t\le T} \Big|\displaystyle\int_{\JJ }\displaystyle\int_{\lbar{t} } ^t\ln(1+z_jg_{ij}(v_2(s)))\tilde{N_j}(ds,dz_j) \Big|^{2p}\right)\,.   
 \end{eqnarray*} 
 By the Burkholder--Davis--Gundy inequality, we have 
\begin{eqnarray}
&& \mathbb{E}\sup_{0\le t\le T} \Big|\displaystyle\int_{\JJ }\displaystyle\int_{\lbar{t}} ^t\ln(1+z_jg_{ij}(v_2(s)))\tilde{N_j}(ds,dz_j) \Big|^{2p} \nonumber\\
 & &\qquad \quad \le \mathbb{E}\sup_{0\le t\le T} \left(\displaystyle \int_{\JJ } \int_{\lbar{t}} ^t
 \Big| \ln(1+z_jg_{ij}(v_2(s))) \Big|^2 \nu_j(dz_j)ds \right)^p\nonumber\\
 & &\qquad \quad \le K_p \De^p\,. \label{a22}
 \end{eqnarray}  
Plugging above, \eqref{a34}, in \eqref{a33} we   get for some $K,K_1,K_2>0$
\begin{eqnarray}
&& \mathbb{E}\sup_{0\le t\le T}\Big| S_i^{\pi}(t) -v_i(t)\Big|^p \leq K_1 \De^p+K_2\De^{p/2} \leq K\De^{p/2}.
\end{eqnarray}
This proves the lemma.
\end{proof}
\begin{theorem}
Assume that Assumptions (A1)–(A4) are true. 
Then, there is a constant $K_{pd,T}$, independent of $\pi$ such that
\begin{eqnarray}
&& \mathbb{E}\Bigg[\sup_{0 \leq t \leq T}\Big[|S(t) - S^{\pi}(t)|^p\Big]\Bigg] \leq K_{pd,T}\De^{p/2}.
\nonumber\\ \end{eqnarray}
\end{theorem}
\begin{proof}
First, we we want to bound 
\begin{eqnarray}
I_1:= \mathbb{E}\Big(\sup_{0 \leq t \leq r}|p(t) - p^{\pi}(t)|^p\Big) \,. 
\end{eqnarray}
From \eqref{e.3.4a}, we see that when $t\in [t_k, t_{k+1}]$, 
\[
\tilde F(S(t-b))=\int_{t_k}^t F(S(u-b))du +O(\Delta^2)\,. 
\]
Thus
\[
\exp\left(\tilde F(S(t-b))\right)=I+\int_{t_k}^t F(S(u-b))du +O(\Delta^2)\,. 
\]
Thus we have a formula for $p(t)$ which is analogous to the one for $p^\pi(t)$ (Equation 
\eqref{ndim3} ): 
 \begin{eqnarray} 
p  (t) 
 &=& \Big[I+ \int_{\lbar{t}}^t  F(S(u-b))du + O(\Delta^2) \Big]\prod_{k=0}^
 {\lfloor t\rfloor} \Big[ I+\int_{t_k}^{t_{k+1}} F(S(u-b))du +O(\Delta^2) \Big]\nonumber\\
 &=& \rho(\lbar{t}, t)\prod_{k=0}^{\lfloor t\rfloor}  \rho(t_k, t_{k+1})\,,
 \end{eqnarray}
 where
 \[
 \rho(r, s)=I+ \int_{r}^s  F(S(u-b))du + O(\Delta^2)
 \,.
 \]
 We can also write
 \begin{eqnarray} 
p^\pi   (t) 
 &=& \rho^\pi(\lbar{t}, t)\prod_{k=0}^{\lfloor t\rfloor}  \rho^\pi (t_k, t_{k+1})\,,
 \end{eqnarray}
where
 \[
 \rho^\pi (r, s)=I+    F(S^\pi (\lbar{s}-b)) (s-r)  
 \,.
 \]
 When $r,s\in [t_k, t_{k+1}], r<s$, we have by the Lipschitz condition
\begin{eqnarray}
 |\rho(r, s)-\rho^\pi(r, s)|
 &\le &    | F(S(t_k-b)) -F(S^\pi (t_k-b)) |  (s-r) \nonumber\\ 
 & &+\int_r^s |
 F(S(u-b))-F(S^\pi (t_k-b))|du+ O(\Delta^2)\nonumber\\ 
 &\le& C|S(t_k-b)  -S^\pi (t_k-b) |+O(\Delta^{3/2} )\,. 
 \end{eqnarray} 
We also have
\begin{equation} 
| \rho^\pi(r, s)|= |I+    F(S^\pi (\lbar{s}-b)) (s-r)
\le |I+C(s-r)|  \le e^{C(s-r)}\,.  
\end{equation} 
In the same way we have
\begin{equation} 
| \rho (r, s)|   \le e^{C(s-r)}\,.  
\end{equation} 
Thus
\begin{eqnarray}
|p^\pi   (t) -p(t)|
&\le & \left| \rho (\lbar{t}, t)-\rho^\pi(\lbar{t}, t)\right|\prod_{k=0}^{\lfloor t\rfloor}  \rho^\pi (t_k, t_{k+1})\nonumber\\
&&\qquad + \sum_{\ell=0}^{\lbar{t}} \left| \rho (t_\ell , t_{\ell+1})-\rho^\pi(t_\ell , t_{\ell+1})\right|
 \rho (\lbar{t}, t) \prod_{k=0, k\not=\ell 
}^{\lfloor t\rfloor}  \rho^\pi (t_k, t_{k+1})\nonumber\\ 
&\le & \left[C\left|S(t_k-b)  -S^\pi (t_k-b) \right|+O(\Delta^{3/2} )  \right]
\prod_{k=0}^{\lfloor t\rfloor} e^{C (  t_{k+1}-t_k)} \nonumber\\
&&\qquad + \sum_{\ell=0}^{\lbar{t}} \left[ C|S(t_\ell -b)  -S^\pi (t_\ell
-b) |+O(\Delta^{3/2} )  \right]
 \rho (\lbar{t}, t) \prod_{k=0, k\not=\ell 
}^{\lfloor t\rfloor} e^{C (  t_{\ell+1}-t_\ell )} \nonumber\\ 
\end{eqnarray}
 Thus  we have for some $C>0$
\begin{eqnarray} 
&& I_1 \leq C \mathbb{E}\sup_{0\le t\le r}\Big|S (t-b)-S^{\pi} (t-b)\Big|^p+ K_1\mathbb{E}\sup_{0\le t\le r}\Big|v_2(u)-S^{\pi} (t-b)\Big|^p. \nonumber
\end{eqnarray}
Then by lemma \ref{a26} we have
\begin{eqnarray} 
&& I_1  \leq C \mathbb{E}\sup_{0\le t\le r}\Big|S (t-b)-S^{\pi} (t-b)\Big|^p+ C \De^{p/2}. \label{a29}
\end{eqnarray} 
We now bound $\mathbb{E}\sup_{0 \leq t \leq r}|X(t)-X^{\pi}(t)|^p$.
Denote 
\begin{eqnarray}
   && A_{i,t} = \sum_{j=1}^d\sum_{\substack{0 \leq u \leq t\\ \De Z(u) \neq 0}} \ln\Big(1+ g_{ij}(S (u-b))Y_{j,N_j(u)} \Big)\nonumber\\&& 
   A^{\pi}_{i,t} = \sum_{j=1}^d\sum_{\substack{0 \leq u \leq t\\ 
   \De Z(u) \neq 0}} \ln\Big(1+ g_{ij}(v_2(u))Y_{j,N_j(u)} \Big) 
\end{eqnarray} 
and denote  $I_2 = \mathbb{E}\Big(\sup_{0 \leq t \leq r}|X(t) - X^{\pi}(t)|^p\Big)$. 
Then,
\begin{eqnarray*}&&
I_2 = \mathbb{E}\Big(\sup_{0 \leq t \leq r}|X(t) - X^{\pi}(t)|^p\Big)\\
&& \leq\Big(\mathbb{E} \sup_{0 \leq t \leq r}\sum_{i=1}^d\Big|\sum_{\substack{0 \leq u \leq t\\ \De Z(u) \neq 0}} \sum_{j=1}^d\big[\ln(1 +g_{ij}(S (u-b))Y_{j,N_j(u)} )-
\nonumber \ln(1 +g_{ij}(v_2(u))Y_{j,N(u)})\\
\\
&&\qquad  +\displaystyle\int_0^t (f_{ii}( S(u-b))-f_{ii}( v_2(u)))du )\big]\Big|^{2p}\Big)\Big)^{1/2}  \Big(\mathbb{E}\Big(|\exp(A_{i,t})+\exp(A^{\pi}_{i,t})|^{2p}\Big)\Big)^{1/2}
\nonumber
\\
&&= \Big(\Big(\sum_{i=1}^d \mathbb{E}\sup_{0 \leq t \leq r}\Big|\displaystyle\int_{\JJ \times [0,t]}\sum_{j=1}^d\left[\ln(1+z_jg_{ij}(S(u-b)))-\ln(1+z_jg_{ij}(v_2(u)))\right] \tilde{N}_j(du,dz)\nonumber\\
&&
+\displaystyle\int_{\JJ \times [0,t]}\sum_{j=1}^d\left[\ln(1+z_jg_{ij}(S (u-b)))-\ln(1+z_jg_{ij}(v_2(u)))\right]\nu_j(dz)du\\&&+\displaystyle\int_0^t (f_{ii}( S(u-b))-f_{ii}( v_2(u)))du\nonumber\Big|^{2p}\Big)\Big)^{1/2}\cdot \Big(\mathbb{E}\Big(|\exp(A_{i,t})+\exp(A^{\pi}_{i,t})|^{2p}\Big)\Big)^{1/2}.\nonumber
\end{eqnarray*}
Then for some $C_1>0$ we have
\begin{eqnarray}
 I_2&&\leq \Big[\Big(C_1 \mathbb{E}\sup_{0 \leq t \leq r}\Big|\displaystyle\int_{\JJ \times [0,t]}\sum_{j=1}^d\left[\ln(1+z_jg_{1j}(S (u-b)))-\ln(1+z_jg_{1j}(v_2(u)))\right] \tilde{N}_j(du,dz_j)\Big|^{2p}\Big)^{1/2} \nonumber\\ 
 &&
+\Big(C_1 \mathbb{E}\sup_{0 \leq t \leq r}\Big|\int_{\JJ\times [0,t]}\sum_{j=1}^d\left[ \ln(1+z_jg_{1j}(S (u-b)))-\ln(1+z_jg_{1j}(v_2(u)))\right] \nu_j(dz_j)du\Big|^{2p}\Big)^{1/2} \nonumber\\ 
 &&
+\Big(C_1 \mathbb{E}\sup_{0 \leq t \leq r}\Big|\displaystyle\int_0^t (f_{ii}( S(u-b))-f_{ii}( v_2(u)))du\Big|^{2p}\Big)^{1/2}\Big]\nonumber
\\&&\cdot \Big(\mathbb{E}\Big(|\exp(A_{1,t})+\exp(A^{\pi}_{1,t})|^{2p}\Big)\Big)^{1/2}
\nonumber\\&&  =: C_1(I^{1/2}_{21} + I^{1/2}_{22}+I^{1/2}_{23})\cdot \Big(\mathbb{E}\Big(|\exp(A_{1,t})+\exp(A^{\pi}_{1,t})|^{2p}\Big)\Big)^{1/2}.\nonumber
\end{eqnarray}
Using the Lipschitz condition on $g_{ij}$, $\int_{\JJ}z_j\nu_j(dz_j) = K_{\nu}< \infty $ for $j=1,2 \cdots, d$, lemma \eqref{a26} and Assumption 3 we have 
\begin{eqnarray*}
&&I_{22}   \leq \mathbb{E}\sup_{0 \leq t \leq r}\Big|\int_{\JJ\times [0,t]}\sum_{j=1}^d\left[ \ln(1+z_jg_{ij}(S (u-b)))-\ln(1+z_jg_{ij}(v_2(u)))\right] \nu_j(dz_j)du\Big|^{2p}\Big)\\
&&\qquad 
\leq C\mathbb{E}\sup_{0 \leq t \leq r}\Big| S (t-b)-S ^{\pi}(t-b)\Big|^{2p}+ C \mathbb{E}\sup_{0 \leq t \leq r}\Big| v_2(u)-S^{\pi} (t-b)\Big|^{2p}\\
&&=: C \mathbb{E}\sup_{0 \leq t \leq r}\Big| S(t-b)-S^{\pi}(t-b)\Big|^{2p} +C\De^p.
\end{eqnarray*}
Using the Burkholder-Davis-Gundy inequality we have
\begin{eqnarray*}
&& I_{21} \\&&\leq \sum_{i=1}^d \mathbb{E}\Big(\displaystyle\int_{\JJ }\displaystyle\int_0^t\sum_{j=1}^d\Big|\ln(1+z_jg_{ij}(S(u-b)))-\ln(1+z_jg_{ij}(v (u-b)))\Big|^2 \nu_j(dz)du\Big)^{p}.
\end{eqnarray*}
Similar to the bound for $I_{22}$ we   have  
\begin{eqnarray*}
&& I_{21} \leq 
C\mathbb{E}\sup_{0 \leq t \leq r}\Big| S (t-b)-S^{\pi} (t-b)\Big|^{2p}+C \De^p.
\end{eqnarray*}
Similar to the bound for $I_{22}$ using assumption (A2) we have
\begin{eqnarray*}
&& I_{23} \leq 
C\mathbb{E}\sup_{0 \leq t \leq r}\Big| S (t-b)-S^{\pi} (t-b)\Big|^{2p}+C \De^p.
\end{eqnarray*}

Combining the bounds for $I_{21}, I_{22}, I_{23}$ with the help of lemma \eqref{a19}, we get for some $K_2>0$
\begin{equation}
    I_{2} \leq  K_2\Big(\mathbb{E}\sup_{0 \leq t \leq r}\Big| S (t-b)-S^{\pi} (t-b)\Big|^{2p}\Big)^{1/2}+K_2\De^{p/2}. \label{a28}
\end{equation}
 We write  $I_3= \mathbb{E}\Big(\sup_{0 \leq t \leq r}|S(t) - S^{\pi}(t)|^p\Big)$.  Then  we have
 \begin{eqnarray*}&&
 I_3 = \mathbb{E}\Big(\sup_{0 \leq t \leq r}|S(t) - S^{\pi}(t)|^p\Big)
 \leq \mathbb{E}\Big(\sup_{0 \leq t \leq r}\Big|(p(t)-p^{\pi}(t))X(t) -  (X(t) - X^{\pi}(t))p^{\pi}(t) \Big|^p\Big)\nonumber \\&&
\leq 2^{p-1}\mathbb{E}\Big(\sup_{0 \leq t \leq r}\Big| (p(t)-p^{\pi}(t) )X(t)\Big|^p\Big) 
+2^{p-1}\mathbb{E}\Big(\sup_{0 \leq t \leq r}\Big|p^{\pi}(t)(X(t) - X^{\pi}(t)) \Big|^p\Big). \nonumber \\&&
=: C(I_{31}+ I_{32}).
 \end{eqnarray*}
We now bound $I_{31},I_{32}$
\begin{eqnarray}
&& I_{31} \leq  C\Big(\mathbb{E}\Big(\sup_{0 \leq t \leq r}\Big| X(t)\Big|^{2p}\Big)\Big)^{1/2}\Big( \mathbb{E}\Big(\sup_{0 \leq t \leq r}\Big| (p(t)-p^{\pi}(t) )\Big|^{2p}\Big) \Big)^{1/2}.
\end{eqnarray} 
Using the Lemmas \ref{a19} and \ref{a29} we will have f 
\begin{eqnarray}&&
I_{31}  
\leq C \Big(\mathbb{E}\sup_{0\leq t\leq r}\Big|S (t-b)-S^{\pi} (t-b)\Big|^{2p}+\De^p\Big)^{1/2}. \nonumber
\end{eqnarray}
Using assumption 2 we can show that $p^{\pi}$ is bounded, hence we can write using \eqref{a28}    
\begin{eqnarray}
I_{32} &\leq& C  \Big(\Big(\mathbb{E}\sup_{0 \leq t \leq r}\Big| S (t-b)-S^{\pi} (t-b)\Big|^{4p}\Big)^{1/2}+\De^p \Big)^{1/2}.
\end{eqnarray}
Hence we have  for some $K_3>0$
\begin{eqnarray}
&& I_{3}\leq K_{3}\Big(\mathbb{E}\sup_{0 \leq t \leq r}\Big|S (t-b)-S^{\pi} (t-b)\Big|^{2p}\Big)^{1/2}+ K_{3}\De^{p/2} \label{a27}.
\end{eqnarray} 
Therefore we get   
\begin{eqnarray}
&&
\mathbb{E}\Bigg[\sup_{0 \leq t \leq r}\Big[|S(t) - S^{\pi}(t)|^p\Big]\Bigg] \nonumber\\
&& 
 \leq C \Big(\mathbb{E}\sup_{0\leq t \leq r}\Big|S (t-b)-S^{\pi} (t-b)\Big|^{2p}\Big)^{1/2}+  K\De^{p/2}. \label{a30}
\end{eqnarray}
Taking $r=b$, we have
\begin{eqnarray} 
\mathbb{E}\Bigg[\sup_{0 \leq t \leq b}\Big[|S(t) - S^{\pi}(t)|^p\Big]\Bigg] \le C \De^{p/2} 
\end{eqnarray}
for any $p\ge 2$.   Now, taking $r=2b$ in \eqref{a30}, we have
\begin{eqnarray}
 &&
 \mathbb{E}\Bigg[\sup_{0 \leq t \leq 2b}\Big[|S(t) - S^{\pi}(t)|^p\Big]\Bigg]
 \le C \left[ \mathbb{E} \sup_{-b\le t\le b} \left|S(t))- S^{\pi}(t) \right| ^{2p}\right]^{1/2} +K \De^{p/2} \nonumber\\
 &&\qquad \le C \left[ K \De^p \right]^{1/2} +K\De^{p/2} \le C \De^{p/2}\,.   
\end{eqnarray} 
Continuing this way, we obtain for any positive integer $k\in \NN$,
\begin{eqnarray}
 &&
 I_{0 \leq t \leq kb }
  \le C_{p, k, d, T} \De^{p/2}\,.   
\end{eqnarray} 
Now, since $T$ is finite, we can choose a $k$ such that
$(k-1)b<T\le kb$.   This completes the proof of the theorem. 
\end{proof}


\end{document}